\documentclass[11pt]{article}
\usepackage{amssymb}
\usepackage[latin1]{inputenc}
\usepackage{graphicx,color}
\def\nd{\noindent}

\usepackage{color}
\newtheorem{theorem}{Theorem}[section]

\newtheorem{lemma}{Lemma}[section]
\newtheorem{proposition}{Proposition}[section]
\newtheorem{remark}{Remark}[section]

\newtheorem{corollary}{Corollary}[section]
\newcommand{\fim}{\hfill\rule{2mm}{2mm}}

\date{}
\begin{document}
\title{
\vspace{0.5in} {\bf\Large Necessary and sufficient conditions for existence of Blow-up solutions for  elliptic  problems in Orlicz-Sobolev spaces}}

\author{
{\bf\large Carlos Alberto Santos}\footnote{Carlos Alberto Santos acknowledges
the support of CAPES/Brazil Proc.  $N^o$ $2788/2015-02$,}\,\, ~~~~~~ {\bf\large Jiazheng Zhou}\footnote{Jiazheng Zhou was supported by CNPq/Brazil Proc. $N^o$ $232373/2014-0$, }\ \ \ 
\hspace{2mm}\\
{\it\small Universidade de Bras\'ilia, Departamento de Matem\'atica}\\
{\it\small   70910-900, Bras\'ilia - DF - Brazil}\\
{\it\small e-mails: csantos@unb.br,
jiazzheng@gmail.com }\vspace{1mm}\\
{\bf\large  Jefferson Abrantes Santos}\footnote{Jefferson Abrantes was Santos Partially supported by CNPq-Brazil grant Casadinho/Procad 552.464/2011-2.}\ \ \ 
\hspace{2mm}\\
{\it\small Universidade Federal de Campina Grande, Unidade Acad\^emica de Matem\'atica}\\
{\it\small  58109-970, Campina Grande - PB - Brazil}\\
{\it\small e-mails:  jefferson@mat.ufcg.edu.br }\vspace{1mm}}

\date{}
\maketitle \vspace{-0.2cm}

\begin{abstract}
This paper is principally devoted to revisit the remarkable works of Keller and Osserman  and generalize some previous results related to the those for the class of quasilinear elliptic problem
$$
 \left\{
\begin{array}{l}
 {\rm{div}} \left( \phi(|\nabla u|)\nabla u\right)  = a(x)f(u)\quad \mbox{in } \Omega,\\
u\geq0\ \ \mbox{in}\ \Omega,\ \ u=\infty\ \mbox{on}\ \partial\Omega,
\end{array}
\right.
$$
where either $\Omega \subset \mathbb{R}^N$ with  
$N \geq 1$ is a smooth bounded domain or $\Omega = \mathbb{R}^N$. The function $\phi$  includes  special cases appearing  in mathematical models in  nonlinear elasticity, plasticity, generalized Newtonian fluids, and in quantum physics.  The proofs are based on comparison principle, variational methods and topological arguments on the Orlicz-Sobolev spaces.
\end{abstract}

\nd {\it \footnotesize 2012 Mathematics Subject Classifications:} {\scriptsize  35A15,  35B44.  35H30 }\\
\nd {\it \footnotesize Key words}: {\scriptsize Orlicz-Sobolev spaces, Blow up solutions, Quasilinear equations.}

\section{Introduction}
\def\theequation{1.\arabic{equation}}\makeatother
\setcounter{equation}{0}

In this paper, let us consider the problems
\begin{equation}\label{P}
 \left\{
\begin{array}{l}
 \Delta_\phi u=a(x)f(u)\ \ \mbox{in}\ \Omega,\\
u\geq0\  \mbox{in}\ \Omega,\  u=\infty\ \mbox{on}\ \partial\Omega,
\end{array}
\right.
\end{equation}
 and \begin{equation}\label{P1}
 \left\{
\begin{array}{l}
 \Delta_\phi u=a(x)f(u)\ \ \mbox{in}\ \mathbb{R}^N,\\
u>0\  \mbox{in}~\mathbb{R}^N,\  u(x)\stackrel{\left|x\right|\rightarrow \infty}{\longrightarrow} \infty,
\end{array}
\right.
\end{equation}
where $\Omega \subset \mathbb{R}^N$ with  $N\geq 1$ is a smooth bounded domain,  $f$ is a continuous function that satisfies $f(0)=0$, $f(s)>0$ for $s>0$, the assumptions on $a(x)$  will be fixed later on, and $\Delta_{\phi}u=div(\phi(|\nabla u|)\nabla u)$ is 
called the $\phi\!-\!Laplacian$ operator, where the $C^1$-function $\phi:(0,+\infty)\to (0,+\infty)$ satisfies:
\begin{enumerate}
\item [\bf{$(\phi)_1$:}] $(\phi(t)t)'>0$ for all $t>0$;
\item [\bf{$(\phi)_2$:}] there exist $l, m>1$ such that $$l\leq\frac{\phi(t)t^2}{\Phi(t)}\leq m~\mbox{for all}~t>0,$$
where $\Phi(t)=\int_0^{|t|}\phi(s)sds$, $t \in \mathbb{R}$;
\item [\bf{$(\phi)_3$:}] there exist $l_1,\ m_1>0$ such that
$$l_1\leq\frac{\Phi''(t)t}{\Phi'(t)}\leq m_1~\mbox{for all}~t>0.$$
\end{enumerate}

Under the {$(\phi)_1$}-{$(\phi)_3$} hypotheses, we have an wide class of $\phi-Laplacian$ operators, for instance: 
\begin{enumerate}
\item [(1)] $\displaystyle\phi(t)=2$, $t>0$. So, $\Delta_{\phi}u=\Delta u$  is the $Laplacian$ operator,
\item [(2)] $\displaystyle\phi(t)=p \vert t \vert^{p-2}$, $t>0$ and $p>1$. In this case, $\Delta_{\phi}u=\Delta_{p}u$  is called the $p\!-\!Laplacian$ operator,
\item [(3)] $\displaystyle\phi(t)=p \vert t \vert^{p-2} + q  \vert t \vert^{q-2}$, $t>0$ and $1<p<q$. The $\Delta_{\phi}u=\Delta_{p}u + \Delta_{q}u$ is called as $(p\&q)\!-\!Laplacian$ operator and it appears in quantum
physics \cite{bfp},
\item [(4)] $\displaystyle\phi(t)= 2\gamma (1+t^2)^{\gamma-1}$, $t>0$ and $\gamma >1$. With this $\phi$, the $\Delta_{\phi}$ operator models problems in nonlinear elasticity problems \cite{naru},
\item [(5)] $\displaystyle\phi(t)= \gamma\frac{(\sqrt{ (1+t^2)}-1)^{\gamma-1}}{\sqrt{1+t^2}}$, $t>0$ and $\gamma \geq 1$. It appears in models of nonlinear elasticity. See, for instance, \cite{dac} for $\gamma =1$ and \cite{naru1} for $\gamma>1$,
\item [(6)] $\displaystyle\phi(t)= \frac{p t^{p-2}(1+t)\ln(1+t) + t^{p-1}}{1+t}$, $t>0$ and $(-1 + \sqrt{1+4N})/2>1$ appears in plasticity problems \cite{naru}.
\end{enumerate} 

 The problems (\ref{P}) (blow-up on the boundary) and (\ref{P1}) also model problems that appear in the theory of automorphic functions, Riemann surfaces, population dynamics, subsonic motion of a gas, non-Newtonian fluids, non-Newtonian filtration as well as in the theory of the electric potential in a glowing hollow metal body.

Researches related to Problem (\ref{P}) was initiated with the case $\phi(t)=2$, $ a = 1$, and $f(u) = exp(u)$ by Bieberbach \cite{bieb} (if $N = 2$) and Rademacher \cite{radema} (if $N = 3$). Problems of this type arise in Riemannian geometry, namely if a Riemannian metric of the form $|ds|^2 = exp(2u(x))|dx|^2$ has constant Gaussian curvature $-c^2$, 
then $\Delta u = c^2 exp(2u)$. Lazer and McKenna \cite{lm} extended the results of Bieberbach and Rademacher for bounded domains in $\mathbb{R}^N$ satisfying a uniform external sphere condition and for exponential-type nonlinearities.

Still for $\phi(t)\equiv 2$, a remarkable development in the study of problem  (\ref{P}) is due to Keller \cite{K} and Osserman \cite{O} that in 1957 established necessary and sufficient conditions for existence of solutions for the problems
$$
{(I)}:~~
\left\{
\begin{array}{l}
 \Delta u=f(u)\ \ \mbox{in}\ \Omega,\\
u\geq 0\  \mbox{in}\ \Omega,\  u=\infty\ \mbox{on}\ \partial\Omega, 
\end{array}
\right.
~~~~~~{(II)}:~~\left\{
\begin{array}{l}
 \Delta u=f(u)\ \ \mbox{in}\ \mathbb{R}^N,\\
u\geq 0\  \mbox{in}\ \Omega,\  u(x)\stackrel{\left|x\right|\rightarrow \infty}{\longrightarrow} \infty,
\end{array}
\right.
$$
where $f$ is a non-decreasing  continuous function. Keller established that
\begin{eqnarray}
\label{deff}
\displaystyle\int_1^{+\infty} \frac{dt}{\sqrt{F(t)}}<+\infty,~\mbox{where}~F(t)=\int_0^tf(s)ds,~t>0
\end{eqnarray}
is a  {sufficient} condition for the problem $(I)$ to have a solution and $(II)$ to have no solution. Besides this, Keller showed that $(II)$ has radially symmetric solutions if, and only if, $f$ does not satisfies (\ref{deff}). In this same year, Osserman proved this same result for sub solutions in $(II)$. 
After these works, (\ref{deff}) has become well-known as the $Keller-Osserman$ condition for $f$.

 The attention of researchers has turned in considering non-autonomous potentials in $(I)$ and $(II)$, that is, particular operators of (\ref{P}). One goal has been to enlarge the class of terms $a(x)$ that still assures existence or nonexistence of solutions for particular cases of (\ref{P}). Another branch of attention of researchers has been in the direction to extend the class of operators in the
problems $(I)$ or $(II)$. For instance, Mohammed in \cite{mo} and Radulescu at all in \cite{radulescu} have considered the problem (\ref{P}) with $\phi(t)=p\vert t \vert^{p-2}$, $t>0$ with $p>1$ and $\Omega$ a bounded domain. Recently, Zhang in \cite{zhang} considered the $p(x)\!-\!Laplacian$ operator in the problem 
(\ref{P}). About $\Omega=\mathbb{R}^N$, the problem (\ref{P}) was considered in \cite{drissi} and references therein with $\phi(t)=p\vert t \vert^{p-2}$, $t>0$ and $2 \leq p \leq N$.

Before doing an overview about these classes of problems, we set that a solution of (\ref{P}) (or (\ref{P1})) is a non-negative function $u\in C^1(\Omega)$ (or a positive function $u\in C^1(\mathbb{R}^N))$ such that $u=\infty$ on $\partial\Omega$, that is, $u(x) \to \infty$ as $d(x)=\inf\{\Vert x-y \Vert~/~y \in \partial\Omega\}\to 0$ (or $u \to \infty$ as $\vert x \vert \to \infty$) and
$$\int_{\Omega}\phi(|\nabla u|)\nabla u\nabla{\psi} dx+\int_{\Omega}a(x)f(u)\psi dx=0,$$
holds for all $ \psi\in C_0^{\infty}(\Omega)$ (or $ \psi\in C_0^{\infty}(\mathbb{R}^N)$).
\smallskip

\noindent{\it Overview about $(\ref{P})$.} 
For $\phi(t)\equiv 2$, the question of existence of solutions to (\ref{P}) was investigated in \cite{ab,ba} with $0< a \in C(\overline{\Omega})$  and $f$ a Keller-Osserman function. In \cite{lair}, Lair showed that (\ref{deff}) is a necessary and sufficient condition for $(\ref{P})$ to have a solution under the more general hypothesis on $a$ than previous paper, namely, $a$ satisfying
\begin{enumerate}
\item [\bf{(a):}] $a:\overline{\Omega}\to[0,+\infty)$ is a $c_{\Omega}-positive$ continuous function,
that is, if $a(x_0)=0$ for some  $x_0\in\Omega$, then there exists an open set $O_{x_0}\subset\Omega$ such that $x_0\in O_{x_0} $ and  
$a(x)>0$ for all $x\in\partial {O_{x_0}}$.
\end{enumerate}

Now for $\phi(t)=\vert t \vert^{p-2}$, $t>0$, we believe that the issue of existence of solutions for (\ref{P}) was first studied in \cite{diaz}. There the term $a$ was considered equal $1$. After this work, a number of important papers have been considering issues as existence, uniqueness and asymptotic behavior for different kinds of weight $a$ and nonlinearities  $f$. See, for instance, \cite{mat}, \cite{mo}, \cite{Garcia-Melian},  \cite{radulescu}, and references therein.

Our objective related to the problem (\ref{P}) is two-fold: first, we generalise previous results to the $\phi\!-\!Laplacian$ operator whose appropriate setting is the Orlicz-Sobolev space. The lack of homogeneity of the operator becomes our estimatives very delicate. Another purpose of this work is to establish necessary and sufficient conditions on the term $f$ for existence of solutions for (\ref{P}) for $a$ in a class of potentials given satisfying $\bf{(a)}$. This result is new even for the context of $p$-Laplacian operator.

To do this, let us consider
\smallskip

{\noindent($\underline{f}$): $\displaystyle\liminf_{s\to+\infty}({\inf\{f(t)~/~ t\geq s\}}/{f(s)})>0,$~~~~~~~~
 ($\overline{f}$): $\displaystyle\limsup_{s\to+\infty}({\sup\{f(t)~/~ 0 \leq t\leq s\}}/{f(s)})<\infty$}
\smallskip

\noindent and note that these assumptions are necessary because we require no kind of monotonicity on $f$, that is, if for instance we assume that $f$ is non-decreasing, as is made in the most of the prior works, the $\underline{f}$ and $\overline{f}$ hypotheses immediately are true. 

Besides this, let us consider
\begin{enumerate}
\item [\bf{(F):}]  $\displaystyle\int_1^{+\infty} \frac{dt}{\Phi^{-1}(F(t))}<+\infty$
\end{enumerate}
and  refer to it {\it as $f$ satisfying the $\phi\!\!-\!\!Keller\!\!-\!\!Osserman$ condition} in reason from the $Keller-Osserman$ condition as it is known for the particular case $\phi(t)=2$.

Ou first result is.
 
\begin{theorem}\label{teo}
 Assume that $(\phi)_1-(\phi)_3$ hold and $a$ satisfies $\bf{(a)}$. Then:
 \begin{enumerate}
 \item [$(i)$] if $f$ satisfies $(\underline{f})$ and $\bf{(F)}$, then Problem $(\ref{P})$ admits at least one solution,
 \item [$(ii)$] if $f$ satisfies $(\overline{f})$ and does not satisfy $\bf{(F)}$, then Problem $(\ref{P})$ have no solution.
 \end{enumerate}
\end{theorem}

To highlight the last theorem and emphasize the importance of hypothesis $(F)$, we state.

\begin{corollary}\label{corol}
 Assume that $(\phi)_1-(\phi)_3$ hold, $a$ satisfies $\bf{(a)}$, and $f$ is a non-increasing function. Then Hypothesis $(F)$ is a necessary and sufficient condition for the problem $(\ref{P})$ has a solution.
\end{corollary}

We note that the above Corollary requires that $f$ goes to infinity at infinity in a strong way when $a$ is bounded in $\Omega$. Yet, if instead of this we permit that $a(x)$ goes to infinity at boundary of $\Omega$ in a strong way, we still can have solutions to (\ref{P}) for $f$ going to infinity at infinity in a slowly way. These was showed in \cite{bandle} or in \cite{mohammed} for $\phi(t) \equiv 2$, $\Omega=B_1(0)$ and $a(x)$ a symmetric radially function and  $f$ does not satisfying the Keller-Osserman condition. So, we can infer that $a(x)$ at boundary of $\Omega$ and $f$ at infinity must behave in a inverse way sense to assure existence of solution. The exactly way is an interesting  open issue.
\smallskip

\noindent{\it Overview about $(\ref{P1})$.}
In $\mathbb{R}^N$, the last researches have showed that the existence of solutions for (\ref{P1}) depends on how $a(x)$ goes to $0$ at infinity and $f$ goes to infinity at infinity. In general, it is known that for either ``fast velocities for both" or ``slow velocities for both" are sufficient conditions for (\ref{P1}) has solutions.   Interesting open issues are ``how should behave $a(x)$ and $f$ at infinity in a inverse way to ensure existence of solutions yet?".

Besides this, the existence of solutions for (\ref{P1}) is sensible to the  another measure related to $a(x)$, more exactly, ``how radial is $a(x)$ at infinity?". To understand this and state our results, let us introduce
$$a_{osc}(r)= \overline{a}(r) - \underline{a}(r),~r\geq 0,$$
where
$$\underline{a}(r)=\min\{a(x)~/~ |x|= r\}~~\mbox{and}~~ \overline{a}(r) =\max\{a(x)~/~ |x| = r\},~r\geq 0,$$
and note that $a_{osc}(r)= 0$, $r\geq r_0$ if, and only if, $a$ is symmetric radially in $\vert x \vert \geq r_0$, for some $r_0\geq 0$.

For $\phi(t)\equiv 2$, Lair and Wood in \cite{lairwood} considered $a$ being a  symmetric radially continuous function (that is,  $a_{osc}(r)=0$ for all $r \geq 0$), $f(u)=u^{\gamma}$, $u\geq 0$ with $0 < \gamma \leq 1$ (that is, $f$ does not satisfies $\bf{(F)})$ and showed that problem (\ref{P}) has a solution if, and only if, 
\begin{equation}
\label{conda}
\int_1^{\infty} r a(r) dr = \infty.
\end{equation}

Still in this context, in 2003, Lair \cite{lair3} enlarged the class of potentials $a(x)$ by permitting $a_{osc} $ to assume not identically null values, but not too big ones. More exactly, he assumed
\begin{equation}
\label{condb}\int_0^{\infty}r a_{osc}(r) exp(\underline{A}(r)) dr< \infty,~~\mbox{where}~~ \underline{A}(r)= \int_0^r s\underline{a}(s) ds,~r\geq 0 \end{equation}
and proved that (\ref{P1}) with suitable $f$ that includes $u^{\gamma}$, $0< \gamma \leq 1$, has a solution if, and only if, (\ref{conda}) holds with $\underline{a}$ in the place of $a$.

Keeping us in this context, Mabroux and Hansen in \cite{hansen} improved the above results by consider 
$$\int_0^{\infty}r a_{osc}(r) (1 + \underline{A}(r))^{\gamma/(1-\gamma)} dr< \infty$$
in the place of (\ref{condb}). In the line of the previous results, recently, Rhouma and Drissi \cite{drissi} considered $\phi(t)=\vert t \vert^{p-2}, ~t \in \mathbb{R}$ with $2 \leq p \leq N$, $f$ a differentiable function that includes $u^{\gamma}$ with $0<\gamma \leq 1$, and they established necessary and sufficient conditions for existence 
of solutions for (\ref{P1}) around the term $a$.

Our objective for the problem (\ref{P1}) is a little bit different of the above ones, because we are principally concerned in establishing necessary and sufficient conditions for existence of solutions of (\ref{P1}) around the function $f$, once fixed a potential $a(x)$ in a bigger class than above results. Besides this, we generalise the prior existence results to the context of $\phi-Laplacian$ operator. 

To do this, first we note that  $(\phi)_1$ and 
$(\phi)_2$  permit us to consider $h^{-1}:(0,\infty) \to (0,\infty)$ being the inverse of $ h(t)=\phi(t)t,~t>0.$ So, let us assume
\begin{enumerate}
\item[$\bf{(A_{\rho})}$:]  $\displaystyle\int_1^\infty h^{-1}(\mathcal{A}_{\rho}(s))dr=\infty,~\mbox{where}~\mathcal{A}_{\rho}(s)=s^{1-N}\int_0^st^{N-1}\rho(t)dt,~s>0$
\end{enumerate}
for suitable continuous function $\rho:[0,\infty) \to [0,\infty)$ given, and $\mathcal{F}:(0,\infty) \to (0,\infty)$, defined by $\mathcal{F}(t)=\frac{t}{2}f(t)^{{-1}/{l_1}}$, be a non-increasing and bijective function such that
\begin{enumerate}
\item [\bf{($\mathcal{F}$):}]  $\mathcal{F}$ is invertible and
$$0\leq\overline{H}:=\displaystyle\int_0^\infty\eta_4(\mathcal{A}_{a_{osc}}(t))h^{-1}\Big (f\Big(\mathcal{F}^{-1}
\Big(\int_0^s h^{-1}(\mathcal{A}_{\overline{a}}(t))dt\Big)\Big)\Big)ds<\infty.$$
\end{enumerate}

After this, we state our existence result.

\begin{theorem}\label{coro53}
Assume that $(\phi)_1-(\phi)_3$ hold. Suppose that  $a(x)$  is a non-negative 
function satisfying $\bf{(A_{\underline{a}})}$, $f$ is a non-decreasing function that does not satisfies $\bf{(F)}$ and such that  $\bf{(\mathcal{F})}$ holds. If 
\begin{equation}
\label{hipo0}
h^{-1}(s + t) \leq  h^{-1}(s ) + h^{-1}(t),~\mbox{for all}~ s,t\geq 0,
\end{equation}
 then  there exists a solution $u\in C^1(\mathbb{R}^N)$ of the problem $(\ref{P1})$ satisfying $\alpha\leq u(0) \leq (\alpha + \varepsilon)+ \overline{H}$, for each $\alpha,\varepsilon>0$ given.
\end{theorem}
\begin{remark} About the above hypotheses, we have:
\item [$i)$] the hypothesis $(\ref{hipo0})$ is satisfied for $\Delta_p$-laplacian operator with $2 \leq p < \infty$,
\item [$ii)$] the hypothesis $\bf{(\mathcal{F})}$ is trivially satisfied for a radial functions $a(x)$ at infinity. In particular, if $a$ symmetric radially, then $\overline{H}=0$, 
\item [$iii)$] as showed in $\cite{hansen}$ for the particular case $\phi(t)=2$, we can not ensure that the problem $(\ref{P1})$ has solution, if we remove the hypothesis  $\bf{(\mathcal{F})}$,
\item [$iii)$]  If $f(s)=s^{\gamma}$, $s>0$ with $0<\gamma< l_1$, then $\mathcal{F}(s)= s^{(l_1 -\gamma)/l_1}/2$, $s>0$.
\end{remark}

Our next goal is to establish necessary conditions, around on $f$, for existence of solutions for (\ref{P1}), once fixed a potential $a$ satisfying  $(A_{\overline{a}})$. To do this, let us consider the problem
\begin{equation}
\label{prob2}
\left\{
\begin{array}{c}
 \Delta_{\phi}u=\overline{a}(|x|)f(u)~\mbox{in}~ \mathbb{R}^N,\\
u\geq0~\mbox{in}~ \mathbb{R}^N,\ u(x)\stackrel{\left|x\right|\rightarrow \infty}{\longrightarrow} \infty,
\end{array}
\right.
\end{equation}
and denote by $A=\sup\mathbb{A}$, where
$$\mathbb{A}=\{\alpha>0~/~ (\ref{prob2})\ \mbox{has a radial solution with}\ u(0)=\alpha\}.$$

So, we have.
\begin{theorem}\label{teo52}
Assume that $(\phi)_1-(\phi)_3$ hold, $a(x)$ is a non-negative continuous function $($not necessarily radial$)$ satisfying  $\bf{(A_{\overline{a}})}$, and $f$ is a non-decreasing function. If $(\ref{P1})$ admits a positive solution, then   $A>0$ and $(0,A)\subset\mathbb{A}$. Moreover, $A=\infty$ if, and only if, $f$ does not satisfies  $\bf{(F)}$.
\end{theorem}

To become clearest our two last Theorems, let us restate them as below.

\begin{corollary}
 Assume that $(\phi)_1-(\phi)_3$ hold. Suppose that  $a(x)$  is a non-negative symmetric radially
function satisfying  $\bf{(A_{a})}$ and $f$ is a non-decreasing function. Then problem $(\ref{P1})$ admits a sequence of symmetric radial solutions $u_k(\vert x \vert)\in C^1(\mathbb{R}^N)$ with $u_k(0)\to\infty$ as 
$k\to\infty$ if, and only if,
  $f$ does not satisfies $\bf{(F)}$. Besides this, $u_{k}^{\prime} \geq 0$ in $[0,\infty)$.
\end{corollary}

\begin{remark}
We emphasize that the equivalence stated in above Corollary is sharp in the sense that we can have solutions for $(\ref{P1})$ with $f$ satisfying  $\bf{(F)}$ and  $a(x)$ as in $\bf{(A_{a})}$. In particular, this shows that $A<\infty$ may occur as well. We quote Gladkov and Slepchenkov $\cite{gladkov}$, where an example was build for the case $\phi(t)=2,~t>0$.
\end{remark}

This paper is organized in the following way. In the section 2, we present the Orlicz and Orlicz-Sobolev spaces  where we work variationally some approximated problems. In the section 3, we present some results of some auxiliary problems. In the   section 4, we completed the proof of Theorems \ref{teo} and Corollary \ref{corol} and in section 5 we prove Theorem \ref{coro53} and \ref{teo52}. 

\section{The Orlicz and Orlicz-Sobolev setting }

In this section, we present an overview about Orlicz and Orlicz-Sobolev spaces, that will be the appropriate settings to deal approximated problems of (\ref{P}) in variational 
way, and also we give some technical results which will  be  need later. For more details about Orlicz and Orlicz-Sobolev spaces, see for instance \cite{rao-ren}. 

A function $M: \mathbb{R} \to [0, +\infty)$ is called an $N$-function if it {is} convex, even, $M(t) =0$ if, and only if, $t=0$,
${M(t)}/{t} \to 0$  as $t\to 0$  and  ${M(t)}/{t} \to +\infty$  as $t\to +\infty$. It is well known that an $N$-function $M$ can be rewritten as
\begin{equation}
\label{nfun}
\displaystyle M(t)=\int_0^{\vert t \vert} m(s)ds,~t\in \mathbb{R},
\end{equation}
where $m:[0,\infty)\to [0,\infty)$ is a right derivative of $M$, non-decreasing, right continuous function, $m(0)=0$, $m(s)>0$ for $s>0$, and $\lim_{s\to \infty}m(s)=\infty$.  Reciprocally, if $m$ satisfies the former properties, then $M$ defined in (\ref{nfun}) is an $N$-function.

For an  $N$-function $M$ and an open set $\Omega \subset \mathbb{R}^N$, the Orlicz class is the set of function defined by
\[
K_M(\Omega) = \left\{ u:\Omega \to \mathbb{R}~/~ \mbox{ $u$ is measurable\ and }  \int_{\Omega} M(u(x))\, dx < \infty \right\}
\]
and the vector space $L^{{M}}(\Omega)$ generated by $K_M(\Omega)$ is called  Orlicz space. When $M$ satisfies the
$\Delta_2$-condition, namely, there exists a constant $k>0$ such that
\[
M(2t) \leq kM(t),~\mbox{for all}~ t\geq 0,
\]
the  Orlicz class  $K_M(\Omega)$ is a vector space, and hence equal to $L^{{M}}(\Omega)$. 

Defining the following norm (Luxemburg norm) on $L^{{M}}(\Omega)$ by
\[
|u|_{M}= \inf\left\{ \lambda >0~/~\int_{\Omega} M\left( \frac{u(x)}{\lambda}\right)\, dx \leq 1 \right\},
\]
we have that the space   $\left(L^{{M}}(\Omega), |\cdot|_{M}\right)$ is a Banach space.  The complement function of $M$ is defined by
\[
\tilde{M}(t) = \sup_{s>0}\left\{ ts - M(s)\right\}~(\mbox{note that}~\tilde{\tilde{M}}=M).
\]

 In the spaces $L^{{M}}(\Omega)$ and $L^{\tilde{M}}(\Omega)$ an extension of H\"{o}lder's inequality holds:
\[
\left|\int_\Omega u(x)v(x)\, dx\right| \leq 2|u|_{M}|v|_{\tilde{M}}, ~\mbox{for all}~ u\in L^{{M}}(\Omega)~\mbox{and}~ v\in L^{\tilde{M}}(\Omega).
\]
As a consequence, to every $\tilde{u} \in L^{\tilde{M}}(\Omega)$  there exists a corresponds $f_{\tilde{u}} \in (L^{\tilde{M}}(\Omega))^*$ such that  $$f_{\tilde{u}}(v) = \int_\Omega \tilde{u}(x)v(x)\, dx,~ v \in L^M(\Omega).$$ Thus, we can define the Orlicz norm on the space  $L^{\tilde{M}}(\Omega)$ by
\[
\|\tilde{u}\|_{\tilde{M}}= \sup_{|v|_{M}\leq 1}\int_\Omega \tilde{u}(x)v(x)\, dx,
\]
and, in a similar way, we can define the Orlicz norm $\|\cdot \|_{{M}}$ on $L^{{M}}(\Omega)$. The norms  $|\cdot |_{{M}}$  and $\|\cdot \|_{{M}}$ are equivalent and satisfy
\[
| u |_{{M}} \leq
\|u \|_{{M}} \leq 2| u |_{{M}}.
\]
It is important to detach that the $L^{{M}}(\Omega)$ is reflexive if and only if $M$ and $\tilde{M}$ satisfy the $\Delta_2$-condition and that
\[
\left( L^{{M}}(\Omega),|\cdot |_{{M}}\right)^* = \left( L^{\tilde{M}}(\Omega),\|\cdot \|_{\tilde{M}}\right)\ \mbox{ and } \
 \left( L^{\tilde{M}}(\Omega),|\cdot |_{\tilde{M}}\right)^* =
\left( L^{{M}}(\Omega),\|\cdot \|_{{M}}\right)
\]
are true.

Now, setting 
 $$W^{1,M}(\Omega)=\Big\{u \in L^{{M}}~/~\exists ~v_i \in L^{{M}}(\Omega)~;\!\int u \frac{\partial \varphi}{\partial x_i}dx = -\!\int\! v_i \varphi dx,~\mbox{for}~i=1, \cdots , N~\mbox{and}~\forall~\varphi \in C_0^{\infty}(\Omega)\Big\},$$
 we have that $W^{1,M}(\Omega)= (W^{1,M}(\Omega),\vert \cdot \vert_{1,M} )$ is a Banach space, where 
\begin{equation}
\label{norma}
\vert u \vert_{1,M} = |u|_M + |\nabla u|_M.
\end{equation}
Besides this, it is well known that $W_0^{1,M}(\Omega)$, denoting the completion of $C_0^{\infty}(\Omega)$ in the norm (\ref{norma}), is a  Banach space as well. It is reflexive if and only if $M$ and $\tilde{M}$ satisfy the $\Delta_2$-condition. Still in this case, we have 
 $$\int_{\Omega}M(u) \leq \int_{\Omega}M(2 d\vert \nabla u \vert)~\mbox{and}~\vert u \vert_{M}\leq 2 d \vert \nabla u \vert_{M}~\mbox{for all}~u \in W_0^{1,M}(\Omega),$$
 where $0<d<\infty$ is the diameter of $ \Omega$. This inequality is well known as Poincar\'e's inequality.  As a consequence of this, we have that $\Vert u \Vert: = \vert \nabla u \vert_{M}$ is an equivalent norm to the norm $\vert u \vert_{1,M}$ on $W_0^{1,M}(\Omega)$. From now on, we consider $\Vert \cdot \Vert$ as the norm on $W_0^{1,M}(\Omega)$.
 
Hereafter, let us assume that $(\phi)_1$ and $(\phi)_2$ hold, that is, $\Phi$ is an $N$-function. We   state three lemmas. Some items of two first ones are due to Fugakai et all \cite[ Lemma 2.1]{naru}.

\begin{lemma}\label{lema20}
Suppose $\phi$ satisfies $(\phi)_1$ and $(\phi)_2$. Considerer
$$
\xi_1(t) = \min\{t^l, t^m\},~  \xi_2(t) =  \max\{t^l, t^m\},~ \eta_1(t) =  \min\{t^{1/l}, t^{1/m~}\},~
\eta_2(t) =  \max\{t^{1/l}, t^{1/m}\},~ t\geq 0.
$$
Then,
\begin{enumerate}
\item [$(i)$] $
\xi_1(\rho)\Phi(t) \leq \Phi(\rho t) \leq \xi_2(\rho)\Phi(t),\quad \mbox{for } \rho, t \geq 0,
$
\item [$(ii)$] $\xi_1(|u|_\Phi )\leq \int_{\Omega} \Phi(|u|)\,dx  \leq \xi_2(|u|_\Phi ),\quad \mbox{for } u \in L^\Phi(\Omega),
$
\item [$(iii)$] $
\eta_1(\rho)\Phi^{-1}(t) \leq \Phi^{-1}(\rho t) \leq \eta_2(\rho)\Phi^{-1}(t),\quad \mbox{for } \rho, t \geq 0,
$
\item [$(iv)$] $L^{\Phi}(\Omega)$, $W^{1,\Phi}(\Omega)$ and $W_0^{1,\Phi}(\Omega)$ are reflexives and separable.
\end{enumerate}
\end{lemma}

Now, remembering that $h$ denotes $h(t):=\phi(t)t$, $t>0$, we can show that.
\begin{lemma}\label{lema51}
Suppose $\phi$ satisfies $(\phi)_1-(\phi)_3$. If 
$$
\xi_3(t) = \min\{t^{l_1}, t^{m_1}\},~  \xi_4(t) =  \max\{t^{l_1}, t^{m_1}\},~~\eta_3(t) =  \min\{t^{1/l_1}, t^{1/m_1~}\}
$$
$$\mbox{and}~~\eta_4(t) =  \max\{t^{1/l_1}, t^{1/m_1}\},~ t\geq 0,$$
then
\begin{enumerate}
\item [$(i)$] 
$\xi_3(\rho)h(t) \leq h(\rho t) \leq \xi_4(\rho)h(t), ~\mbox{for } \rho, t \geq 0,$
\item [$(ii)$] $
\eta_3(\rho)h^{-1}(t) \leq h^{-1}(\rho t) \leq \eta_4(\rho)h^{-1}(t), ~\mbox{for } \rho, t \geq 0.
$
\end{enumerate}
\end{lemma}

The next Lemma is very important in our approach to control solutions of approximated problems. 

\begin{lemma}\label{lema21}
Suppose $\phi$ satisfies $(\phi)_1$, $(\phi)_2$  and $B=B(x,t) \in L_{loc}^{\infty}(\Omega \times \mathbb{R})$ is non-decreasing in $t \in \mathbb{R}$. Considerer $u,v \in C^1(\Omega)$ satisfying
$$
 \left\{
\begin{array}{l}
 \Delta_\phi u \geq B(x, u)\ \ \mbox{in}\ \Omega,\\
 \Delta_\phi v \leq B(x, v)\ \ \mbox{in}\ \Omega.
\end{array}
\right.
$$
If $u \leq v$ on $\partial\Omega$ $($that is, for each $\delta>0$ given there exists a neighborhood of $\partial\Omega$ in what $u < v + \delta$$)$, then $u \leq v$ in $\Omega$.
\end{lemma}

\begin{proof}
Since $(\phi)_1$ and $(\phi)_2$ hold, we able  to show that 
\begin{equation}\label{des}
\langle \phi(\vert  \nabla u \vert){\nabla u} - \phi(\vert  \nabla v \vert){\nabla v}, 
{\nabla u}- \nabla v\rangle >0~\mbox{for all}~u,v \in C^1(\Omega)~\mbox{with}~u \neq v
\end{equation} 
holds. So, by Theorem 2.4.1 in \cite{serrin-pucci}, it follows the claim.

\fim
\end{proof}

\section{Auxiliary results}

Our principal strategy to show existence of solutions for (\ref{P}) and (\ref{P1}) is to solve  boundary value problems with finite data, to control these solutions and then to get the solution for (\ref{P}) or (\ref{P1}) by a limit process. To do this, first we prove a variational sub and super solution theorem for the problem
\begin{equation}\label{P0}
 \left\{
\begin{array}{l}
 \Delta_\phi u=g(x,u) \ \mbox{in}\ \Omega,\\~u=k \ \mbox{on}\ \partial\Omega,
\end{array}
\right.
\end{equation}
where $ k\geq 0$ is an appropriate real number and $g:\Omega\times\mathbb{R}\to\mathbb{R}$ is a  $C(\bar{\Omega}\times\mathbb{R})$ function.

We define a sub solution of (\ref{P0}) as being a function
$\underline{u}\in W^{1,\Phi}(\Omega)\cap L^\infty(\Omega)$ such that $\underline{u}\leq k $ on $\partial\Omega$ and
$$\int_{\Omega}\phi(|\nabla \underline{u}|)\nabla \underline{u}\nabla\psi dx+\int_{\Omega}g(x,\underline{u})\psi dx\leq0$$
holds for all $\psi\in W_0^{1,\Phi}(\Omega)$ with $\psi\geq0$ in $\Omega$. A super solution is a function $\overline{u}\in W^{1,\Phi}(\Omega)\cap L^\infty(\Omega)$ satisfying the converse above inequalities. So, a solution is a function ${u}\in W^{1,\Phi}(\Omega)\cap L^\infty(\Omega)$ that is simultaneously a sub and a super solution for (\ref{P0}).

\begin{proposition}\label{pro31}
 Suppose that  $(\phi)_1-(\phi)_3$ hold and $\underline{u},\bar{u}\in W^{1,\Phi}(\Omega)\cap L^\infty(\Omega)$ with $\underline{u} \leq \bar{u} $ are sub e super solutions of $(\ref{P0})$, respectively. Then there exists a $u\in W^{1,\Phi}(\Omega)\cap L^\infty(\Omega)$ with $\underline{u}\leq u\leq\bar{u}$ solution  of $(\ref{P0})$. Besides this, if $g \in L^{\infty}(\Omega \times \mathbb{R})$, then $u\in C^{1,\alpha}(\bar{\Omega})$, for some $0<\alpha<1$.
\end{proposition}
\begin{proof}
 Defining
$$
\mathfrak{h}(x,t)=\left\{
\begin{array}{ll}
g(x,\underline{u}(x)),& \mbox{if}\ t\leq \underline{u}(x)-k,\\
g(x,t+k),& \mbox{if}\ \underline{u}(x)-k\leq t\leq \bar{u}(x)-k,\\
g(x,\bar{u}(x)),& \mbox{if}\ t\geq\bar{u}(x)-k,
\end{array}
\right.
$$
we have that $\mathfrak{h}\in C(\bar{\Omega}\times\mathbb{R})\cap L^{\infty}(\Omega\times\mathbb{R})$. Since,  $(\phi)_1$ and $(\phi)_2$ hold, we have that
$$I(u)=\int_{\Omega}{\Phi}(|\nabla u|)dx+\int_{\Omega}H(x,u)dx, ~u\in W_0^{1,\Phi}(\Omega)$$
is well-defined and $I\in C^1(W_0^{1,\Phi}(\Omega),\mathbb{R})$, where
 $$H(x,t)=\int_0^t\mathfrak{h}(x,s)ds,~t\in \mathbb{R}.$$

Besides this, it follows from Poincar\'e's inequality and Lemma \ref{lema20} that $I$ is coercive, because
$$
\begin{array}{lcl}
I(u)&\geq& \displaystyle\int_{\Omega}\Phi(|\nabla u|)dx-C\int_{\Omega}|u|dx\geq {\xi_1}(\Vert u \Vert)-C|u|_{L^1}\\
 \\
&\geq&\displaystyle\xi_1(\Vert u \Vert)-\bar{C}\Vert u \Vert\to\infty~\mbox{as}~\Vert u \Vert \to \infty,
\end{array}
$$
where $C,\bar{C}>0 $ are real constants. Finally, as $\mathfrak{h}\in  L^{\infty}(\Omega\times\mathbb{R})$, we have that $I$ is  weak s.c.i.
Since $W_0^{1,\Phi}(\Omega)$ is a reflexive space, there exists a
$$u_0\in W_0^{1,\Phi}(\Omega)~~\mbox{such that}~~I^{\prime}(u_0)=0~\mbox{and}~I(u_0)=\min_{v\in W_0^{1,\Phi}(\Omega)}I(v),$$
that is, $u_0$ is a weak solution of 
\begin{equation}\label{PPhi0}
\left\{
\begin{array}{l}
 \Delta_\phi v=\mathfrak{h}(x,v) \ \mbox{in}\ \Omega,\\~v=0 \ \mbox{on}\ \partial\Omega.
\end{array}
\right.
\end{equation}
Now, we show that $\underline{u}-k\leq u_0\leq \bar{u}-k$ a.e. in $\Omega$. In fact,
by taking $\left(u_0-(\bar{u}-k)\right)^+\in W_0^{1,\Phi}(\Omega)$ as a test function and using that $(\bar{u}-k)$ is a super solution of (\ref{PPhi0}), we obtain
\begin{eqnarray*}
\int_{\Omega}\phi(|\nabla u_0|){\nabla u_0}\nabla\left(u_0-(\bar{u}-k)\right)^+dx&=&-\int_{\Omega}\mathfrak{h}(x,u_0)\left(u_0-(\bar{u}-k)\right)^+dx\\
&=&-\int_{\Omega}g(x,\bar{u})\left(u_0-(\bar{u}-k)\right)^+dx\\
&\leq&
\int_{\Omega}\phi({|\nabla\bar{u}|})\nabla\bar{u}\nabla\left(u_0-(\bar{u}-k)\right)^+dx,\end{eqnarray*} that is,

$$\int_{\{u_0>\bar{u}-k\}}\langle\phi(|\nabla u_0|)\nabla u_0-\phi(|\nabla\bar{u}|)\nabla\bar{u},\nabla(u_0-\bar{u})\rangle dx\leq0.$$

So, by using (\ref{des}) we get the claim. In analogue way, we obtain $\underline{u} - k\leq u_0$ in $\Omega$. Setting $u = u_0 + k$, we have that it is a solution of $(\ref{P0})$ with $\underline{u}\leq u \leq \bar{u}$ a.e. in $\Omega$. Finally, the regularity follows from the Lemma 3.3 in \cite{naru}. This ends the proof.
\end{proof}

\fim

Now, setting the reflexive Banach space
$$W_{rad}^{1,\Phi}(B_R(0))=\{u\in W^{1,\Phi}(B_R(0))~/~ u~ \mbox{is symmetric radially}\},$$
where $B_R(x_0)$ stands for the ball in $\mathbb{R}^N$ centred at $x_0$ with radius $R>0$. So, we have the below result.
\begin{corollary}\label{31}
Suppose that  $(\phi)_1-(\phi)_3$ hold and $\underline{u},\bar{u}\in W_{rad}^{1,\Phi}(B_R(0))\cap L^\infty(B_R(0))$, with $\underline{u} \leq \bar{u} $, are sub e super solutions for $(\ref{P0})$, respectively. Then there exists a $u\in W_{rad}^{1,\Phi}(B_R(0))\cap L^\infty(B_R(0))$ with $\underline{u}\leq u\leq\bar{u}$ solution  of $(\ref{P0})$.
\end{corollary}
\begin{proof} Following the same arguments as in the above proof, we obtain  an $u_0\in W_{0,rad}^{1,\Phi}(B_R(0))$, with  $\underline{u}-k\leq u_0\leq\bar{u}-k$, and
$$I(u_0)=\min\{I(v)~/~v\in W_{0,rad}^{1,\Phi}(B_R(0))\},$$
where
$$W_{0,rad}^{1,\Phi}(B_R(0))=\{u\in W_0^{1,\Phi}(B_R(0))~/~ u~ \mbox{is symmetric radially}\}.
$$
So, by the principle of symmetric criticality  \cite{Palais}, we obtain that
$$I^{\prime}(u_0)=0~\mbox{and}~I(u_0)=\min_{u\in {W_0^{1,\Phi}(B_R(0))}}I(u).$$ 

In analogue way, we show that $u=u_0+k$ is a symmetric radially  solution for (\ref{P0}), with $\underline{u}\leq u\leq\bar{u}$.
\end{proof}
\fim

Now, let us emphasize the solution of problem
\begin{equation}\label{PBB}
 \left\{
\begin{array}{l}
\Delta_{\phi}v =cg(v)\ \mbox{in}\ B_L(0),\\
v\geq 0\ \mbox{in}\ B_L(0),\ \ v=k\ \mbox{on}\ \partial B_L(0),
\end{array}
\right.
\end{equation}
where $c,k, L>0$ are real constants given.

\begin{corollary}\label{32}
Suppose that $(\phi)_1$-$(\phi)_3$ hold. If  $g$ is a non-decreasing and continuous function on $[0,\infty)$ such that $g(t)>0$ for $t>0$, then there exists a $v(\vert x \vert)=v_{k,L}(\vert x \vert)\in W_{rad}^{1,\Phi}(B_L(0))\cap C^{1,\alpha}(\bar{B}_L(0))$, for some $0<\alpha<1$, satisfying:
\begin{enumerate}
\item [$(i)$] $0\leq v(0) \leq{v(|x|)} \leq k$, $v^{\prime}(0)=0$, and $v^{\prime}\geq 0$ on $[0, L]$,
\item [$(ii)$] $v_{k,L }\leq v_{k+1,L}$ and $v_{k,L}\geq v_{k,L+1}$ on $[0,L]$,
\item [$(iii)$] the inequalities
\begin{equation}\label{desige}
\eta_2^{-1}\left(\frac{c}{l_1}\right)\int_{v(0)}^r\frac{d\tau}{\Phi^{-1}({G}(v(0),\tau))}\leq r\leq \eta_1^{-1}\left(\frac{c}{m_1 N}\right) \int_{v(0)}^r\frac{d\tau}{\Phi^{-1}
({G}(v(0),\tau))}
\end{equation}
hold for all $0 \leq r \leq L$, where $\eta_1$ and $\eta_2$ were defined at Lemma $\ref{lema20}$ and 
$${G}(x,y):=\int_{x}^{y}{g}(t)dt~\mbox{for}~x,y\in \mathbb{R}~\mbox{with}~0\leq x <y.$$
\end{enumerate}
\end{corollary}

\noindent\begin{proof} By applying the Proposition \ref{pro31} and Corollary \ref{31} with $\underline{u} = 0$ and $\overline{u} = k$ as sub solution and super solution, respectively, we obtain a 
$v(\vert x \vert)=v_{k,L}(\vert x \vert)\in  W_{rad}^{1,\Phi}(B_L(0))\cap C^{1,\alpha}(\bar{B}_L(0))$ solution of (\ref{PBB}). Besides this, by using Lemma \ref{lema21}, we prove $(ii)$.  To show $(i)$ and $(iii)$, let us set
$$
v_{r,\epsilon}(t)=\left\{
\begin{array}{ll}
 1&\  \mbox{if}\ 0\leq t\leq r,\\
\mbox{linear}&\  \mbox{if}\ r\leq t\leq r+\epsilon,\\
0&\  \mbox{if}\  r+\epsilon \leq t \leq L,
\end{array}
\right.
$$
for each $\varepsilon>0$ given such that $0\leq r<r+\epsilon<L$. 

Now, by taking $\psi(x)=v_{r,\epsilon}(|x|)$ as a  test function, we have 
$$-\int_{B_L(0)}\phi(|\nabla v|)\nabla v\nabla\psi dx=c\int_{B_L(0)} {g}(v(x))\psi dx,$$
holds. So,
$$\frac{1}{\epsilon}\int_{A_{r,r+\epsilon}}\phi(|v'|)v' dx=c\int_{B_r(0)}{g}(v)dx+c\int_{A_{r,r+\epsilon}}{g}(v )v_{r,\epsilon}(|x|)dx,$$
where $A_{r,r+\epsilon}=\{x\in\mathbb{R}^N, r\leq|x|\leq r+\epsilon\}$. That is, 
$$\frac{1}{\epsilon}\int_r^{r+\epsilon}t^{N-1}\phi(|v'|)v'dt=c\int_0^rt^{N-1}g(v)dt+
c\int_r^{r+\epsilon}t^{N-1}g(v)v_{r,\epsilon}dt$$
and by taking $\epsilon\to0$, we obtain
\begin{equation}\label{hemeo1}
r^{N-1}\phi(|v'|)v'=c\int_0^rt^{N-1}g(v)dt\geq 0,~0<r<L,
\end{equation}
because $g$ is non-negative, that is, $v'\geq 0$ on $[0,L]$.

Besides this,   it follows from $(\phi)_1$ and $(\phi)_2$ that $h(s)=\phi(s)s$, $s>0$, is a $C^1$-increasing homeomorphism on $[0, \infty)$
such that $h(0)=0$. As $g$ is continuous,  it follows from (\ref{hemeo1}) that
\begin{equation}\label{e1}v^{\prime}\in C^1([0,R))~\mbox{and}~v^{\prime}(r)=h^{-1}\Big(cr^{1-N}\int_0^rt^{N-1} g(v)dt\Big),~0<r<L,
\end{equation}
and in particular, we have $v^{\prime}(0)=0$. This ends the proof of $(i)$. 

To prove $(iii)$, first we note that (\ref{hemeo1}) implies that
$$(\phi(v')v')'+\frac{N-1}{r}\phi(v')v'=cg(v)$$
is holds true for all $0<r<L$. As a consequence of these, we have that 
$$[\Phi'(v'(r))]'\leq cg(v(r)),~\mbox{for all}~0<r<L.$$

Besides this, by using (\ref{hemeo1}), $g$ , and $v$ increasing, we obtain 
\begin{eqnarray*}
[\Phi'(v'(r))]'&=&cg(v)-\frac{N-1}{r}\phi(v')v'\\
&\geq&cg(v)-\frac{N-1}{r}\left(\frac{cr}{N}g(v)\right)\\
&=&\frac{c}{N}g(v(r)),~\mbox{for all}~0<r<L,
\end{eqnarray*}
that is,
$$\frac{1}{N}cg(v(r))\leq [\Phi'(v'(r))]'\leq cg(v(r)),~0<r<L.$$

Now, by using of $v^{\prime}(0)=0$ and $(\phi)_3$, we have 
$$\begin{array}{ccl}
  \displaystyle\frac{c}{N}\int_{v(0)}^{v(r)}g(t)dt&\leq&\displaystyle\int_0^r[\Phi'(v'(t))]'v'(t)dt
\displaystyle=\displaystyle\int_0^r\Phi''(v'(t))v'(t)v''(t)dt\\
\displaystyle &\leq&\displaystyle m_1\int_0^r[\Phi(v'(t))]'dt
=m_1\Phi(v'(r)),~0<r<L
  \end{array}
$$
and, in similar way,
 $$c\int_{v(0)}^{v(r)}g(t)dt\geq l_1 \Phi(v'(r)),~0<r<L,$$
that is,
\begin{equation}\label{imp}
\frac{c}{m_1 N}\int_{v(0)}^{v(r)}g(t)dt\leq \Phi(v'(r))\leq\frac{c}{l_1}\int_{v(0)}^{v(r)}g(t)dt,~0<r<L.
\end{equation}

Now by using the definition of $G(x,y)$, 
we can rewrite (\ref{imp}) as
$$
\Phi^{-1}\Big(\frac{c}{m_1 N}G(v(0),v(r))\Big)\leq v'(r)\leq \Phi^{-1}\Big(\frac{c}{l_1}G(v(0),v(r))\Big),~0<r \leq L,
$$
to obtain
$$\int_{v(0)}^{v(r)}\frac{d\tau}{\Phi^{-1}\big(\frac{c}{l_1}G(v(0),\tau)\big)}\leq r\leq\int_{v(0)}^{v(r)}\frac{d\tau}{\Phi^{-1}
\big(\frac{c}{m_1 N}G(v(0),\tau)\big)},\ \ 0<r\leq L.$$
So, by using the Lemma \ref{lema20}, we obtain (\ref{desige}). This completes the proof of the Corollary.
\end{proof}
\fim

\section{On bounded domain}

Before proving Theorems \ref{teo}, we need of the next result to help us to control a sequence of solutions of an  approximate problem. 
\begin{lemma}\label{teo0}
 Assume that $(\phi)_1-(\phi)_3$ hold. If $\Omega=B_R(0)$,  $a(x)$ is a positive 
symmetric radially function  and $f$ satisfies $(\underline{f})$ and $\bf{(F)}$, then Problem $(\ref{P})$ admits at least one symmetric radially 
solution $u(\vert x \vert) \in C^1(B_R(0))$  such that $u^{\prime}(0)=0$ and $u^{\prime}\geq 0$ for $0<r<R$. 
\end{lemma} 

\noindent\begin{proof} By  applying Corollary \ref{31}, with $\underline{u}= 0$ and  $\overline{u}= k-1$, we obtain an $u_k\in  W_{rad}^{1,\Phi}(B_R(0))\cap C^{1,\alpha}(\bar{B}_R(0))$, for some $\alpha \in (0,1)$, satisfying
\begin{equation}\label{PB2}
 \left\{
\begin{array}{c}
\Delta_{\phi}u_k=a(\vert x \vert){f}(u_k)\ \mbox{in}\ B_R(0),\\
u_k\geq 0\ \mbox{in}\ B_R(0),\ \ u_k=k-1\ \mbox{on}\ \partial B_R(0),
\end{array}
\right.
\end{equation}
for each $k \in \mathbb{N}$ with $k\geq 2$ given. Besides this, it follows from Lemma \ref{lema21}, that
\begin{eqnarray}
\label{desb}
0\leq u_1 \leq u_2 \leq \cdots \leq u_k \leq \cdots~\mbox{in}~\overline{B}_R(0),
\end{eqnarray}
and a consequence of this, it  there exists 
$$0\leq u(x):=\lim_{k\to \infty} u_k(x)\leq \infty,~\mbox{for each}~x \in B_R(0).$$

In the following, let us show that $0\leq u(x)<\infty$, for each $x \in B_R(0)$, and it is a solution of (\ref{P}). To do this, given an  $x_0 \in B_R(0)$, let  us  denote by  $a_{\infty}=\min\{a(t)~/~0 \leq t = \vert x-x_0 \vert\leq L\}>0$ and apply Corollary \ref{32} with $L=R-\vert x_0 \vert>0$, to obtain a 
$v_k(\vert x - x_0\vert)\in  W_{rad}^{1,\Phi}(B_L(0))\cap C^{1,\alpha}(\bar{B}_L(0))$ that satisfies
\begin{equation}\label{PB1}
 \left\{
\begin{array}{c}
\Delta_{\phi}v_k=a_{\infty}
{f}(v_k)\ \mbox{in}\ B_L(0),\\
v_k\geq 0\ \mbox{in}\ B_L(0),\ \ v_k=k\ \mbox{on}\ \partial B_L(0),
\end{array}
\right.
\end{equation}
for each $k \in \mathbb{N}$ given. As another consequence of Corollary \ref{32} and Lemma \ref{lema21}, we have
\begin{eqnarray}
\label{desbb}
0 \leq u_1 \leq u_2 \leq \cdots \leq u_k \leq v_k \leq v_{k+1}\leq \cdots~\mbox{in}~\overline{B}_R(0),
\end{eqnarray}
that is, there exists $$0 \leq v_{x_0,L}:=\lim_{k\to\infty}v_k(x_0)\leq\infty.$$

We claim that $v_{x_0,L}<\infty.$ In fact, assume that there exists a $k_0$ such that $v_{k_0}(x_0)>0$ (on the contrary, we have nothing to do).  Let us assume $x_0=0$ to simplify our reasoning. So, given $M>1$ and denoting by $F(x,y)=F(y) - F(x)$ for $x,y\in \mathbb{R}$, with $F$ defined at (\ref{deff}), it follows from Lemma \ref{lema20} and  ${\Phi}^{-1}$  non-decreasing, that 
\begin{equation}\label{desiga}\begin{array}{ccl}
  \displaystyle\int_{v_k(0)}^{\infty}\frac{d\tau}{\Phi^{-1}
({F}(v_k(0),\tau))}&\leq&\displaystyle\int_{0}^{\infty}\frac{d\tau}{\Phi^{-1}
({F}(v_k(0),v_k(0)+\tau ))}\\
\displaystyle &\leq&\displaystyle \int_{0}^{1}\frac{d\tau}{\Phi^{-1}
({F}(v_k(0),v_k(0)+\tau ))} + \int_{1}^{M}\frac{d\tau}{\Phi^{-1}
({F}(v_k(0),v_k(0)+\tau ))}\\
&+ & \displaystyle\int_{M}^{\infty}\frac{d\tau}{\Phi^{-1}
({F}(v_k(0),v_k(0)+\tau ))}\\
& \leq& \displaystyle\frac{1}{{\Phi^{-1}
({f}(v_k(0)))}}  \int_{0}^{1}\frac{d\tau}{\tau^{1/l}}+ \displaystyle\frac{1}{{\Phi^{-1}
({f}(v_k(0) ))}}  \int_{1}^{M}\frac{d\tau}{\tau^{1/m}}\\
&+& \displaystyle\int_{M}^{\infty}\frac{d\tau}{\Phi^{-1}
({F}(0,\tau))}\\
&=& \displaystyle\frac{1}{{\Phi^{-1}
({f}(v_k(0)))}}\Big[ \frac{l}{l-1} - \frac{m}{m-1} + \frac{m}{m-1} M^{\frac{m-1}{m}}\Big]\\
&+& \displaystyle\int_{M}^{\infty}\frac{d\tau}{\Phi^{-1}({F}(\tau ))},
  \end{array}
\end{equation}
where we used in the fourth inequality that
$${F}(v_k(0),v_k(0)+\tau )\geq {F}(0,\tau )={F}(\tau )~\mbox{and}~{F}(v_k(0),v_k(0)+\tau )\geq {f}(v_k(0) )\tau~\mbox{for all}~\tau \geq 0,$$
because ${f}$ is non-decreasing.

Now, supposing by contradiction that $v_{x_0,L} = \infty$, it follows from (\ref{desige}) and (\ref{desiga}), that
$$\begin{array}{ccl}
0<L&\leq& \displaystyle\eta_1^{-1}\left(\frac{a_\infty}{m_1 N}\right) \lim_{k\to \infty}\int_{v_k(0)}^k\frac{d\tau}{\Phi^{-1}
({F}(v_k(0),\tau))} \\
& \leq & \displaystyle \eta_1^{-1}\left(\frac{a_\infty}{m_1 N}\right) \lim_{k\to \infty}\int_{v_k(0)}^{\infty}\frac{d\tau}{\Phi^{-1}
({F}(v_k(0),\tau))} \leq 
\displaystyle\eta_1^{-1}\left(\frac{a_\infty}{m_1 N}\right)\int_{M}^{\infty}\frac{d\tau}{\Phi^{-1}
({F}(\tau ))}.
\end{array}
$$
for all $M>1$ given. So, doing $M \to \infty$, we obtain a contradiction by using the hypothesis $\bf{(F)}$.

So, we showed that $u(x)=\lim_{k\to\infty}u_k(x)$ for each $x \in B_R(0)$ is well defined and by using standard arguments 
we are able to prove that
 $u\in C^{1}({B}_R(0))$ is a symmetric radially solution for (\ref{P}). Besides this, by using
similar arguments as those used to show that $v^{\prime}\geq 0$ in Corollary \ref{32}, we can show that $u^{\prime}\geq 0$
as well.  These complete the proof of Theorem \ref{teo0}.
\end{proof}
\fim
\smallskip

Now, we prove the Theorem \ref{teo}, by using the Lemma \ref{teo0}.
\smallskip

\noindent\begin{proof} $of$ ($i$): By applying the Proposition \ref{pro31}, we get a $u_k\in C^{1,\alpha}(\bar{\Omega})$ for some $0<\alpha<1$ that satisfies $u_k\leq u_{k+1}$ and
\begin{equation}\label{k}
 \left\{
\begin{array}{c}
 \Delta_\phi u_k=a(x)f(u_k)\ \mbox{in}\ \Omega,\\
u_k{\geq}0\ \mbox{in}\ \Omega,\ u_k=k\ \mbox{on}\ \partial\Omega,
\end{array}
\right.
\end{equation}
for all $k \in \mathbb{N}$ given. Now, given a $x_0 \in \Omega$, let us consider two cases. 

First, assume $a(x_0)>0$. Then,  there exists a neighborhood  $V\subset\Omega$ of $x_0$ such that $a(x)>0$ for all $ x\in \bar V$, that is, $\mathfrak{m}=\min_{\bar{V}}a>0$. Besides this, it follows from the hypothesis ($\underline{f}$) that the  non-increasing function
\begin{equation}\label{def}
\underline{f}(t)=\inf\{f(s),\ s\geq t\},\ t\geq0
\end{equation}
 is well defined and fulfils: $0 \leq \underline{f}\leq f$, $\underline{f}(s)=0$ if, and only if $s=0$, satisfies the  $\phi\!\!-\!\!Keller\!\!-\!\!Osserman$ condition. In fact, let $\tau=\liminf_{t\to+\infty}\{\underline{f}(t)/f(t)\}$. It follows from ($\underline{f}$), there exists a $M>0$ such that $\underline{f}(t)\geq \tau f(t)$ for all $t>M$. So,  
$$\underline{F}(t)=\int_0^t\underline{f}(s)ds=\int_0^M\underline{f}(s)ds+\int_M^t\underline{f}(s)ds\geq\frac{\tau}{2}F(t), ~t \geq M,$$
that is, by  Lemma \ref{lema20}, we have
 $$\Phi^{-1}(\underline{F}(t))\geq \Phi^{-1}(\frac{\tau}{2}F(t))\geq \eta_1(\tau /2)\Phi^{-1}(F(t)), ~t \geq M.$$
These show our claim.

Now, let us consider the problem
\begin{equation}\label{V}
 \left\{
\begin{array}{c}
 \Delta_\phi v=\mathfrak{m}\underline{f}(v)\ \mbox{in}\ V,\\
v\geq 0\ \mbox{in}\ V,\ v=\infty\ \mbox{on}\ \partial V,
\end{array}
\right.
\end{equation}
where $\underline{f}$ is defined at (\ref{def}).
So, it follows from Lemmas \ref{teo0} and \ref{lema21} that, there exists a $v \in C^1(\Omega)$ solution of (\ref{V}) that satisfies   $0\leq u_k\leq u_{k+1}\leq v$ in $V$. 

Now, if  $x_0\in \Omega$ is such that $a(x_0)=0$, then it follows from $c_{\Omega}-positive$ hypothesis under $a$ that there exists a neighborhood $V\subset\Omega$ of $x_0$ such that
$a(x)>0$ for all $x\in\partial V$. By the compactness  of  $\partial V$, there exist open sets $V_i,\ i=1,\dots,n$ such that 
$$\partial V\subset \cup_{i=1}^nV_i\ \mbox{and}\ a(x)>0,~\mbox{for all}~ x\in V_i.$$

So by above arguments, we obtain that $u_k\leq v_i$ in $V_i$, where $v_i\in C^1(V_i)$ is a solution of
\begin{equation}\label{Vi}
 \left\{
\begin{array}{c}
 \Delta_\phi v=\mathfrak{m_i}\underline{f}(v)\ \mbox{in}\ V_i,\\
v\geq 0\ \mbox{in}\ V_i,\ v=\infty\ \mbox{on}\ \partial V_i,
\end{array}
\right.
\end{equation}
and $\mathfrak{m_i}=\min_{\bar{V}_i}a>0$. Hence,  there exists a real constant $C=C_V>0$ such that $0 \leq u_k\leq C$ in $\partial V$. Again, by the Lemma  \ref{lema21}, we obtain  $u_k\leq C$ in $V$. That is, in both cases, $u_k$ is bounded locally on $\Omega$. So, by standard arguments, we can show that $u_k\to u\in C^1(\Omega)$ and $u$ is a solution of ($\ref{P}$). 
\smallskip

\noindent $Proof$ $of$ ($ii$): Denoting by 
$$\mathfrak{M}=\max_{\overline{\Omega}}a>0~\mbox{and}~ \overline{f}(t)=\sup\{f(s),\ 0 \leq s \leq t\},\ t\geq0,$$
it follows from the hypotheses $(a)$ and $(\overline{f})$ that: $M>0$, $\overline{f}(0)=0$,  $\overline{f}(t)>0$ for $t>0$, $\overline{f}$
 is non-decreasing function, and $\overline{f}$ does not satisfy $\bf{(F)}$.

Now, assume by contradiction, that problem (\ref{P}) admits a solution $u\in C^1(\Omega)$. So, it follows from above informations that $u$ satisfies
\begin{equation}
 \left\{
\begin{array}{c}
 \Delta_{\phi}u= a(x)f(u) \leq \mathfrak{M}\overline{f}(u)\ \mbox{in}\ \Omega,\\
u\geq 0\ \mbox{in}\ \Omega,\ u=\infty\ \mbox{on}\ \partial \Omega.
\end{array}
\right.
\end{equation}

On the other hand, it follows from Proposition \ref{pro31}, that there exists a $u_k\in C^1(\overline{\Omega})$
satisfying
\begin{equation}
 \left\{
\begin{array}{l}
 \Delta_{\phi}u_k=\mathfrak{M} \overline{f}(u_k)\ \mbox{in}\ \Omega,\\
u_k \geq 0\ \mbox{in}\ \Omega,\ u_k=k\ \mbox{on}\ \partial \Omega,
\end{array}
\right.
\end{equation}
and
 $$0\leq u_1\leq u_2\leq \dots\leq u_k\leq \dots\leq u,$$
as a consequence of Lemma \ref{lema21}. So, there exists an $\omega\in C^1(\Omega)$ such that $u_k\to \omega$ in $C^1(\Omega)$, $0 \leq \omega \leq u$, and
\begin{equation}\label{116}
 \left\{
\begin{array}{c}
 \Delta_{\phi}\omega=\mathfrak{M} \overline{f}(\omega)\ \mbox{in}\ \Omega,\\
\omega\geq 0\ \mbox{in}\ \Omega,\ \omega=\infty\ \mbox{on}\ \partial \Omega.
\end{array}
\right.
\end{equation}

Finally, considering $R>0$ such that $\overline{\Omega} \subset \subset B_R(0)$, it follows from Corollary \ref{31} and (\ref{desige}) that there exists a $w_k \in C^1(\overline{B}_R(0))$ with $w^{\prime}_k\geq 0$ satisfying 
\begin{equation}\label{117}
 \left\{
\begin{array}{c}
 \Delta_{\phi}w_k=\mathfrak{M} \overline{f}(w_k)\ \mbox{in}\ B_R(0),\\
w_k \geq0\ \mbox{in}\ B_R(0),\ w_k=k\ \mbox{on}\ \partial B_R(0),
\end{array}
\right.
\end{equation}
and
$$
\eta_2^{-1}\left(\frac{\mathfrak{M}}{l_1}\right)\int_{w_k(0)}^k\frac{d\tau}{\Phi^{-1}(\overline{F}(\tau))}\leq\eta_2^{-1}\left(\frac{\mathfrak{M}}{l_1}\right)\int_{w_k(0)}^k\frac{d\tau}{\Phi^{-1}(\overline{F}(\tau)-\overline{F}(w_k(0)))}\leq R,
$$
for each $k>0$ given, where $\overline{F}(s)=\int_0^s \overline{f}(t)dt,~s\geq 0$. 

Since, $\overline{f}$ does not satisfies $\bf{(F)}$, there exists a $k_0>0$ such that $w_{k_0}(0)> \min\{\omega(x)~/~x\in \Omega\}$. This is impossible, because $w_{k_0} \leq \omega$ in $\Omega$, by using the Lemma \ref{lema21}. This completes the proof of Theorem \ref{teo}.
\end{proof}
\fim
 
\section{On  whole $\mathbb{R}^N$}

We begin this section with the below Lemma which is a radial version of Theorem \ref{coro53}. We emphasize that in it we do not require the hypotheses 
$\bf{(\mathcal{F})}$ and $(\ref{hipo0})$. So, let us consider
\begin{equation}
\label{prob22}
\left\{
\begin{array}{c}
 \Delta_{\phi}u=\rho(|x|)f(u)~\mbox{in}~ \mathbb{R}^N,\\
u>0~\mbox{in}~ \mathbb{R}^N,\ u(x)\stackrel{\left|x\right|\rightarrow \infty}{\longrightarrow} \infty,
\end{array}
\right.
\end{equation}
and state our below lemma.

\begin{lemma}\label{teo51}
 Assume that $(\phi)_1-(\phi)_3$ hold, $\rho$ is a non-negative continuous function 
satisfying  $\bf{(A_{\rho})}$, and $f$ is a non-decreasing function such that  $\bf{(F)}$ is not satisfied. Then $\mathbb{A_{\rho}}= (0,\infty)$, where $\mathbb{A_{\rho}}=\{\alpha>0~/~ (\ref{prob22})\ \mbox{has a radial solution with}\ u(0)=\alpha\}$. 
\end{lemma}
\begin{proof}
 Given $\alpha>0$, consider the problem
\begin{equation}\label{a0}
 \left\{
\begin{array}{l}
 (r^{N-1}\phi(|u'|u'))'=r^{N-1}{\rho}(r)f(u(r)),\ \ r>0\\
u'(0)=0,\ \ u(0)=\alpha.
\end{array}
\right.
\end{equation}

Since ${\rho}$ and $f$ are continuous, we can follow the arguments in \cite{haitao}, to conclude that there exist a  $\Gamma(\alpha)>0$ (maximal extreme to the right for the existence interval of solutions for (\ref{a0}) and a  $u_\alpha\in C^2(0,\Gamma(\alpha))\cap C^1([0,\Gamma(\alpha)))$ solution of (\ref{a0}) on $(0,\Gamma(\alpha))$. If we had  $\Gamma(\alpha)<\infty$ for some $\alpha>0$, then we would have, by ordinary differential equations theory, that  $u_\alpha(r)\to\infty$ as $r\to\Gamma(\alpha)^-$.

So, $u_\alpha(|x|)$ would be a symmetric radially solution of the problem 
$$
\left\{
\begin{array}{c}
 \Delta_{\phi}u={\rho}(|x|)f(u)~\mbox{in}~ B_{\Gamma(\alpha)}(0),\\
u\geq0~\mbox{in}~ B_{\Gamma(\alpha)}(0),\ u=\infty~\mbox{on}~ \partial B_{\Gamma(\alpha)}(0),
\end{array}
\right.
$$
but this is impossible by Theorem \ref{teo}, because $f$ does not satisfy $\bf{(F)}$.

Besides this, it follows from $f$  non-decreasing,  Lemma \ref{lema51} and $\bf{(A_{\rho})}$, that   
$$
\begin{array}{ccl}
 u_\alpha(r)&\geq&\alpha+\displaystyle\int_0^rh^{-1}\Big(f(\alpha)s^{1-N}\int_0^st^{N-1}{\rho}(t)dt\Big)ds\\
&\geq&\alpha+\eta_3(f(\alpha))\displaystyle\int_0^rh^{-1}\Big(s^{1-N}\int_0^st^{N-1}{\rho}(t)dt\Big)ds\to\infty,~\mbox{as}~r\to\infty,
\end{array}
$$
that is, $u_\alpha(|x|)$, $x\in\mathbb{R}^N$ radial solution of 
$$
\left\{
\begin{array}{c}
 \Delta_{\phi}u=\rho(|x|)f(u)~\mbox{in}~ \mathbb{R}^N,\\
u\geq \alpha~\mbox{in}~ \mathbb{R}^N,\ u(x)\stackrel{\left|x\right|\rightarrow \infty}{\longrightarrow} \infty,
\end{array}
\right.
$$
with $u_\alpha(0)=\alpha$. That is, $\alpha \in \mathbb{A_{\rho}}$. This ends our proof.
\end{proof}
\fim
\smallskip

\noindent\begin{proof}{\it of Theorem $\ref{coro53}$.} Given $\beta>\alpha>0$, it follows from the Lemma \ref{teo51} that there exist positive radial solutions $u_\alpha$ and $u_\beta$  of the problems
$$
\left\{
\begin{array}{l}
 \Delta_{\phi}u=\overline{a}(|x|)f(u)\ \mathbb{R}^N,\\
u_\alpha(0)=\alpha,~ u(x)\stackrel{\left|x\right|\rightarrow \infty}{\longrightarrow} \infty,
\end{array}
\right.~~~~\mbox{and}~~~~\left\{
\begin{array}{l}
 \Delta_{\phi}u=\underline{a}(|x|)f(u)\ \mathbb{R}^N,\\
u_\beta(0)=\beta,~ u(x)\stackrel{\left|x\right|\rightarrow \infty}{\longrightarrow} \infty,
\end{array}
\right.
$$
respectively.

Besides this, it follows from $u_\alpha,~f$ non-decreasing, Lemma \ref{lema51}, and $\bf{(A_{\underline{a}})}$ that
$$
\begin{array}{lll}
 u_\alpha(r)&\leq&\alpha+\displaystyle\int_0^rh^{-1}(\mathcal{A}_{\overline{a}}(t))f(u_\alpha(t))dt\leq 2\displaystyle\eta_4(f(u_\alpha(r)))\int_0^r h^{-1}
(\mathcal{A}_{\overline{a}}(t))dt\\
&\leq&2\displaystyle(f(u_\alpha(r)))^{1/l_1}\int_0^r h^{-1}(\mathcal{A}_{\overline{a}}(t))dt,
\end{array}
$$
for all $r>0$ sufficiently large. That is, 
\begin{equation}
\label{estimativa}
u_\alpha(r)\leq \mathcal{F}^{-1}\Big(\int_0^r h^{-1}(\mathcal{A}_{\overline{a}}(t))dt\Big),~\mbox{for all}~r>>0.
\end{equation}

Now, setting 
$$0<S(\beta)=sup\{r>0 ~/~u_\alpha(r)<u_\beta(r)\} \leq \infty,$$
for $\alpha>0$ given, we claim that $S(\beta)=\infty$ for all $\beta>\alpha + \overline{H}$. In fact, by assuming this is not 
true, then there exists a $\beta_0>\alpha + \overline{H}$ such that $u_\alpha(S(\beta_0))=u_\beta(S(\beta_0))$. So, by using that $f$ is non-increasing and $u_{\alpha} \leq u_{\beta}$ on $[0,S(\beta_0)]$, we obtain that  
\begin{equation}
\label{wh}
\beta_0 \leq \alpha +\! \int_0^{S(\beta_0)}\!\!\Big[h^{-1}\Big(s^{1-N}\int_0^st^{N-1}\overline{a}(t)f(u_\alpha(t))dt\Big)-
h^{-1}\Big(s^{1-N}\int_0^s \!t^{N-1}\underline{a}(t)f(u_\alpha(t))dt\Big)\Big]ds
\end{equation}
holds.

On the other hands, it follows from $f$, $u_\alpha$  non-decreasing, (\ref{hipo0}), and Lemma \ref{lema51}, that
$$
\begin{array}{lll}
0 &\leq& \Big[h^{-1}\Big(s^{1-N}\int_0^st^{N-1}\overline{a}(t)f(u_\alpha(t))dt\Big)-
h^{-1}\Big(s^{1-N}\int_0^st^{N-1}\underline{a}(t)f(u_\alpha(t))dt\Big)\Big] \chi_{[0,S(\beta)]}(s)\\
&=&\Big[h^{-1}\Big(s^{1-N}\int_0^st^{N-1}[\overline{a}(t)-\underline{a}(t)]f(u_\alpha(t))dt+ h^{-1}\Big(s^{1-N}\int_0^st^{N-1}\underline{a}(t)f(u_\alpha(t))dt\Big)\Big)\\
&-& h^{-1}\Big(s^{1-N}\int_0^st^{N-1}\underline{a}(t)f(u_\alpha(t))dt\Big)\Big] \chi_{[0,S(\beta)]}(s)
\leq h^{-1}\Big(s^{1-N}\int_0^st^{N-1}[\overline{a}(t)-\underline{a}(t)]f(u_\alpha(t))dt\Big)\\
&\leq& h^{-1}(\mathcal{A}_{a_{osc}}(s)f(u_\alpha(s))  \leq \eta_4(\mathcal{A}_{a_{osc}}(s)) h^{-1} (f(u_\alpha(s)))\\

&\leq&\eta_4(\mathcal{A}_{a_{osc}}(s)) 

h^{-1}\Big( f\Big(

\mathcal{F}^{-1}\Big(\int_0^r h^{-1}(\mathcal{A}_{\overline{a}}(t))dt\Big)\Big)\Big):
=\mathcal{H}(s) ,~s\geq 0,
\end{array}
$$
where $\chi_{[0,S(\beta)]}$ stands for the characteristic function of $[0,S(\beta)]$.

So, it follows from the hypothesis  $\bf{(\mathcal{F})}$ and  (\ref{wh}), that 
$$\beta_0 \leq \alpha +\int_0^\infty \mathcal{H}(s)ds\leq \alpha+ \overline{H},$$
but this is impossible. 

Now, by setting $\beta=(\alpha + \epsilon) + \overline{H}$, for each $\alpha,\epsilon>0$ given, and considering the problem
\begin{equation}\label{Wn}
 \left\{
\begin{array}{l}
 \Delta_{\phi} u={a}(x)f(u)\ \mbox{in}\ B_n(0),\\
u\geq 0~\mbox{in}~B_n(0),~~u=u_\alpha\ \mbox{on}~ \partial B_n(0),
\end{array}
\right.
\end{equation}
we can infer from  Proposition \ref{pro31} that there exists a $w_n = w_{n,\alpha} \in C^{1,\nu}(\overline{B}_n)$, for some $0<\nu<1$, solution of  (\ref{Wn}) satisfying 
$0<\alpha\leq u_\alpha\leq w_n\leq u_\beta$ in $B_n$ for all $n \in \mathbb{N}$. So, by compactness, there exists a  
$w\in C^1(\mathbb{R}^N)$ such that $w(x)=\lim_{n\to\infty}w_n(x)$ is a  solution of (\ref{P}).\fim
\end{proof}
\smallskip

By adjusting the above arguments, we are able to prove the below remark, which generalises the main result in \cite{sayeb}. 
\begin{remark}
If we assume the stronger hypothesis
$$
 \tilde{H}:=\int_0^\infty[\eta_4(a^*(s))-\eta_3(a_*(s))]h^{-1}\Big (sf\Big(\mathcal{F}^{-1}\Big(\int_0^s
h^{-1}(\mathcal{A}_{a^{*}}(t))dt\Big)\Big)\Big)ds<\infty,
$$
instead of $(\mathcal{F})$,
we obtain the same results of Theorem $\ref{coro53}$ with $\tilde{H}$ in the place of $\overline{H}$, without assuming $(\ref{hipo0})$, where
$$a_*(r)=\min\{a(x);\ |x|\leq r\},~~  a^*(r)=\max\{a(x);\ |x|\leq r\},~r\geq 0.$$
\end{remark}
\smallskip

\noindent\begin{proof}{\it of Theorem $\ref{teo52}$.}  Assume that $w\in C^1(\mathbb{R}^N)$ is a positive solution of (\ref{P}). So, given  $0< \alpha<w(0)$, it follows by the arguments in the proof of Lemma \ref{teo51} that there exists a radial solution $u_\alpha \in C^1(B_{\Gamma(\alpha)}(0))$  of the problem
\begin{equation}\label{BPB}
\left\{
\begin{array}{c}
 \Delta_{\phi}u=\overline{a}(|x|)f(u)~\mbox{in}~ B_{\Gamma(\alpha)}(0),\\
u\geq0~\mbox{in}~ B_{\Gamma(\alpha)}(0),\ u=\infty~\mbox{on}~\partial B_{\Gamma(\alpha)}(0),
\end{array}
\right.
\end{equation}
if $\Gamma(\alpha)<\infty$. Yet, this is impossible by Lemma \ref{lema21}, that is, $\Gamma(\alpha)=\infty$. Besides this, it follows from $\bf{(A_{\underline{a}})}$  that $u_\alpha$ is a solution of the problem (\ref{prob2}), that is, $\alpha \in \mathbb{A}$. In particular, it follows from the above arguments that $(0,A)\subset\mathbb{A}$. 

Finally, if we assume that $f$ satisfies the $\phi-$Keller-Osserman condition, it follows from Theorem \ref{teo} that the problem 
$$
\left\{
\begin{array}{c}
 \Delta_{\phi}u=\overline{a}(|x|)f(u)~\mbox{in}~ B_{1}(0),\\
u\geq0~\mbox{in}~ B_{1}(0),\ u=\infty~\mbox{on}~\partial B_{1}(0),
\end{array}
\right.
$$
admits a solution. So, it follows from Lemma \ref{lema21} that $A<\infty$. Reciprocally, if $A<\infty$, then $u_{A+1}$ is a radial solution of 
$$
\left\{
\begin{array}{c}
 \Delta_{\phi}u=\overline{a}(|x|)f(u)~\mbox{in}~ B_{\Gamma(A+1)}(0),\\
u\geq0~\mbox{in}~ B_{\Gamma(A+1)}(0),\ u=\infty~\mbox{on}~\partial B_{\Gamma(A+1)}(0),
\end{array}
\right.
$$
where $0<\Gamma(A+1)<\infty$. So, by Theorem \ref{teo}, we have that $f$ satisfies $\bf{(F)}$. These ends the proof. \fim
\end{proof}

\begin{center}
\LARGE{Acknowledgement}
\end{center}

This paper was completed while the first author was visiting the Professor Haim Brezis at Rutgers University. He thanks to Professor Brezis by his incentive and hospitality. In this time, the second author was visiting the Professor Zhitao Zhang at Chinese Academy of Sciences. He is grateful by the invitation and 
hospitality as well.


\begin{thebibliography}{00}
\bibitem{radulescu} 
R. Alsaedi, H. M\^aagli, V. R\~adulescu, 
N. Zeddini, {Asymptotic behaviour of positive large solutions of
quasilinear logistic problems},  Electron. J. Qual. Theory Differ. Equ. No. 28 (2015) 1-15.

\bibitem{bandle} C. Bandle, Y. Cheng, G. Porru, { Boundary blow-up in semilinear elliptic problems with singular
weights at the boundary}, Applicable Mathematics (J. C. Misra, ed.), Narosa, New Delhi (2001) 68-81.

\bibitem{ba} C. Bandle, M. Marcus, {Asymptotic behaviour of solutions and their derivatives for semilinear elliptic problems
with blow-up on the boundary}, Ann. Inst. H. Poincar\'e 12 (1995) 155-171.

\bibitem{ab} C. Bandle; M. Marcus, {Large solutions of semilinear elliptic equations: existence, uniqueness and asymptotic
behaviour}, J. Anal. Math. 58 (1992) 9-24.

\bibitem{sayeb} 
N. Belhaj Rhouma, A. Drissi and W. Sayeb, {Nonradial large solutions for a class of nonlinear problems},  
Complex Var. Elliptic Equ. Vol. 59, no 5 (2014) 706-722.

\bibitem{drissi} N. Belhaj Rhouma, A. Drissi, {Large and entire large solutions for a class of nonlinear problems}, 
Appl. Math. Comput. 232 (2014) 272-284. 

\bibitem{bfp} V. Benci; D. Fortunato and L. Pisani, Solitons like solutions
of a Lorentz invariant equation in dimension 3, Rev. Math. Phys., 10
(1998), 315-344.

\bibitem{bieb} L. Bieberbach, $\Delta u = e^u$ und die automorphen Funktionen, Math. Ann. 77 (1916) 173-212. 

\bibitem{dac} B. Dacorogna, Introduction to the Calculus of Variations, ICP London 2004.

\bibitem{diaz} G. Diaz, R. Letelier, {Explosive solutions of quasilinear elliptic equations: existence and
uniqueness}, Nonlinear Anal. 20 (1993), 97-125.

\bibitem{naru1} N. Fukagai, K. Narukawa, Nonlinear eigenvalue problem for a model equation
of an elastic surface. Hiroshima Math. J., 25(1) (1995) 19-41.

\bibitem{naru} N. Fukagai, K. Narukawa, {On the existence of multiple positive solutions
of quasilinear elliptic eigenvalue problems}, Ann. Mat. Pura Appl., 4(3) (2007) 539-564.

\bibitem{Garcia-Melian} J. García-Melián, { Large solution for equations involving the $p$-Laplacian and singular weights},
 Z. Angew. Math. Phys. 60 (2009) 594-607.

\bibitem{gladkov} A. Gladkov, N. Slepchenkov, {Entire solutions of semilinear elliptic equations}, 
Electron. J. Differential Equations, Vol. 2004(2004), No. 86 pp. 1-10. 

\bibitem{K} J. B. Keller, {On solutions of $\Delta u = f(u)$}, Comm. Pure Appl. Math. 10 (1957) 503-510. 

\bibitem{lair} A. Lair, {A necessary and sufficient condition for existence of
large solutions to semilinear elliptic equations}, J. Math. Anal. Appl. 240 (1999) 205-218.

\bibitem{lair3} A. Lair, {Nonradial large solutions of sublinear elliptic equations}, Appl. Anal. 82 (2003) 431-437.

\bibitem{lairwood} A. Lair,  A. Wood, {Large solutions of sublinear elliptic equations}, Nonlinear Anal. 39 (2000) 745-753. 

\bibitem{lm} A.C. Lazer, P.J. McKenna, {On a problem of Bieberbach and Rademacher}
Nonlinear Anal. TMA 21 (1993) 327-335.

\bibitem{hansen} K. Mabrouk, W. Hansen, { Nonradial large solutions of sublinear elliptic}, J. Math. Anal. Appl., 330 (2007) 1025-1041.

\bibitem{mat} J. Matero, {Quasilinear elliptic equations with boundary blow-up}, J. Anal. Math. 69 (1996) 229-247.

\bibitem{mohammed} A. Mohammed, {Boundary behavior of blow-up solutions to some
weighted non-linear differential equations}, Electron. J. Differential Equations, Vol. 2002(2002), No. 78 pp. 1-15.

\bibitem{mo} A. Mohammed, {Existence and asymptotic behavior of blow-up solutions to weighted quasilinear equations}, 
J. Math. Anal. Appl. 298 (2004) 621-637.

\bibitem{O} R. Osserman, {On the inequality $\Delta u \geq f(u)$}, Pacific. J. Math. 7 (1957) 1641-1647.

\bibitem{Palais} R. Palais, {The principle of symmetric criticality}, Commun.Math. Phys. 69 (1997) 19-30.

\bibitem{serrin-pucci} P. Pucci, J. B. Serrin, {The Maximum Principle}, Vol. 73 Springer Science \& Business Media 2007.

\bibitem{radema} H. Rademacher, Einige besondere probleme partieller Differentialgleichungen, in Die Differential- und Zntegraigleichungen
der Mechanik und Physik I, 2nd edition, pp. 838-845. Rosenberg, New York (1943)

\bibitem{rao-ren} M. M. Rao, Z. D. Ren, {Theory of Orlicz Spaces, Marcel Dekker}, New York  1991.

\bibitem{haitao} H. Yang, {On the existence and asymptotic behavior of large solutions for a semilinear elliptic problem in 
$\mathbb{R}^N$}, Commun. Pure Appl. Anal., 4, No 1 (2005) 187-198.

\bibitem{zhang} Q. Zhang, {Existence of blow-up solutions to a class of $p(x)\!-\!Laplacian$ problems}, Int. Journal of Math. Analysis, 
Vol. 1 no. 2 (2007) 79 - 88.
\end{thebibliography}
\end{document}